\documentclass[12pt,twoside,cd]{amsart}

\usepackage[]{geometry}
\usepackage{fancyhdr}
\usepackage{marvosym}

\usepackage{amsmath,amssymb,amsthm,amsrefs,mathtools,tikz, amscd, verbatim, graphicx, xcolor, color, mathrsfs,enumerate,dsfont,extarrows,eucal,floatrow,array}
\usetikzlibrary{shapes,matrix,patterns}
\usepackage[all]{xy}
\usepackage{pdflscape}

\usepackage[english]{babel}
\usepackage{graphicx}

\newcommand{\vect}[1]{\textbf{#1}}
\newcommand{\set}[1]{\left\lbrace #1 \right\rbrace}

\newcommand{\mindeg}[1]{{\text{mindeg } #1}}

\newcommand{\Hom}{\operatorname{Hom}}
\newcommand{\trop}{\operatorname{trop}}
\newcommand{\spec}{\operatorname{Spec}}
\newcommand{\Bl}{\operatorname{Bl}}
\newcommand{\Gl}{\operatorname{Gl}}
\newcommand{\Sl}{\operatorname{Sl}}
\newcommand{\val}{\operatorname{val}}
\newcommand{\init}{\operatorname{in}}

\newcommand{\gr}{\operatorname{gr}}

\newtheorem{theorem}{Theorem}[section]

\newtheorem{corollary}[theorem]{Corollary}

\theoremstyle{definition}

\newtheorem{example}[theorem]{Example}
\newtheorem{definition}[theorem]{Definition}

\usepackage{chngcntr}
\counterwithin{figure}{section}

\title{Tropicalizing Spherical Embeddings}
\author{Evan D. Nash}
\address{Department of Mathematics, The Ohio State University, 231 W 18th Ave, Columbus OH, 43210}
\email{nash.228@osu.edu}

\begin{document}
\begin{abstract}
Recently, a theory for the tropicalization of a spherical homogeneous space $G/H$ was developed by Tassos Vogiannou. We extend his ideas to define the tropicalization of a spherical $G/H$-embedding. This generalizes the construction of tropicalization of toric varieties.
We also extend an alternate definition of spherical tropicalization via Gr\"obner theory developed by Kiumars Kaveh and Christopher Manon and show that the two notions of extended tropicalization coincide.
\end{abstract}
\maketitle
\small
\section*{Introduction}\label{intro}

Marrying algebraic geometric ideas and combinatorics is an active area of research. In recent years, the notion of tropicalization has proven to be a fruitful such tool to answer questions in algebraic geometry. Particularly, toric varieties have benefited from tropical geometric methods meshing with their inherent combinatorial structure. See for example Chapter 6 of \cite{MS} for an introduction to some examples of the utility of tropical ideas in the toric world. Toric varieties are examples of spherical varieties, which encompass a wider class of algebraic objects, among them flag varieties and symmetric varieties.
Spherical varieties also have combinatorial structure in the form of colored fans, which directly generalize the well-known polyhedral fans of toric geometry.
This theory was developed by Luna and Vust \cite{LV} in 1983.
It is a natural idea to take advantage of the similar combinatorial structure and extend the theory of tropicalization from toric varieties to the more general case.

The first steps in this direction were taken by Tassos Vogiannou in his thesis \cite{Vo}.
Among other results, Vogiannou developed a definition for the tropicalization of subvarieties of a spherical homogeneous space, which is the analogue of the dense torus orbit present in a toric variety.
In a forthcoming paper, Kiumars Kaveh and Christopher Manon extend Vogiannou's work by defining a theory of Gr\"obner bases on spherical varieties and showing that their definition agrees with Vogiannou's via a spherical fundamental theorem. 
They further consider a notion of spherical amoebas and show, with an additional assumption, that this amoeba approaches the tropicalization.

The purpose of this note is to define the tropicalization of a general spherical embedding.
Our blueprint for this construction appeared separately in \cite{Ka} and \cite{Pay}.
These papers define the tropicalization of a toric variety by extending the tropicalization of its dense torus.
We will mimic their ideas using the shared polyhedral fan structure of toric and spherical varieties.

We also describe how the Kaveh-Manon method of tropicalization that uses Gr\"obner theory can be extended to spherical embeddings.
Finally, we show that the two notions of extended tropicalization coincide.

The layout of this article is as follows. In \textsection \ref{sphvar}, we review the basic theory of spherical varieties, and \textsection \ref{HomSpaces} describes Vogiannou's definition of tropicalizing spherical homogeneous spaces.
Our construction of the tropicalization of a spherical embedding is explained in \textsection\ref{construction}, and \textsection\ref{examples} contains examples.
In \textsection \ref{FTTG}, we recall the Fundamental Theorem of Tropical Geometry from the toric case and its extension to toric varieties.
Then \textsection \ref{Grobtrop} discusses Gr\"obner tropicalization and \textsection \ref{extGrobtrop} proves an extended fundamental theorem.

\subsection*{Acknowledgments} Much of this work was completed at the Fields Institute in the University of Toronto. The author is indebted to Gary Kennedy for suggesting this area of research and for providing guidance and discussions throughout its development. This paper also benefited from conversations with Kiumars Kaveh, Christopher Manon, and Jenia Tevelev.

\section{Spherical Varieties}\label{sphvar}

There are a number of surveys on spherical varieties and their combinatorial structure.
Refer for example to \cite{LV}, \cite{Kn}, \cite{Pas}, or \cite{Pe} for more details on the theory discussed in this section.
There is occasionally a clash between symbols usually used for toric varieties and their analogs in spherical varieties; whenever possible we have favored the toric conventions as these are more widely known.
We work throughout over an algebraically closed field $k$.
Let $G$ be a connected reductive group with a Borel subgroup $B$. Let $H \leq G$ be a closed subgroup such that the action of $B$ on $G/H$ via the left action of $G$ has an open orbit.
In this case, we call $G/H$ a \emph{spherical homogeneous space}.
A normal $G$-variety $X$ that contains $G/H$ as an open orbit of the action of $G$ is called a \emph{spherical embedding}.

Let $\mathcal{X}$ denote the group of characters $B \rightarrow k^*$.
We consider the $B$ semi-invariant rational functions on $G/H$:
\[
k(G/H)^{(B)} := \set{f \in k(G/H)^\ast \mid \text{there exists } \chi_f \in \mathcal{X} \text{ such that } gf = \chi_f(g)f \text{ for all } g \in B }.
\]
Here, the action of the Borel subgroup on $k(G/H)$ is given by $gf(x) = f(g^{-1}x)$, so $gf$ is only defined on those $x$ such that $g^{-1}x$ is in the domain of $f$.
This affords us a homomorphism $k(G/H)^{(B)} \rightarrow \mathcal{X}$ defined by $f \mapsto \chi_f$.
Further, the kernel of this map is the set of nonzero constant functions, which we write as $k^\ast$.
Then denote by $\mathcal{M}$ or $\mathcal{M}(G/H)$ the image of $k(G/H)^{(B)}/k^\ast$ in $\mathcal{X}$.
The lattice $\mathcal{M}$ is finitely generated and free, so we obtain a vector space $\mathcal{N}_\mathbb{Q}(G/H) := \text{Hom}(\mathcal{M}, \mathbb{Q}) \cong \mathbb{Q}^m$, simply denoted $\mathcal{N}_\mathbb{Q}$ when the underlying homogeneous space is clear.
The integer $m$ is called the \emph{rank} of the $G/H$-embedding.

We must further define the valuation cone, which will lie inside $\mathcal{N}_\mathbb{Q}$. We consider $G$-invariant $\mathbb{Q}$-valuations $k(G/H) \rightarrow \mathbb{Q}$ which are trivial on $k^\ast$. 
By restricting such a valuation to $k(G/H)^{(B)}$, we obtain an induced map $k(G/H)^{(B)}/k^\ast \rightarrow \mathbb{Q}$, so we can identify it with a point in $\mathcal{N}_\mathbb{Q}$.
This identification yields a bijection between the set of such $G$-invariant valuations and a rational convex cone in $\mathcal{N}_\mathbb{Q}$ which we call the \emph{valuation cone}, denoted by $\mathcal{V}(G/H)$ or $\mathcal{V}$.
We write $\mathcal{D}(G/H)$ or $\mathcal{D}$ for the (finite) set of $B$-stable prime divisors in $G/H$. We refer to $\mathcal{D}(G/H)$ as the \emph{palette} of $G/H$ and call its elements \emph{colors}.
Every color $D$ induces a valuation $\nu_D$ on $k(G/H)$ given by a function's order of vanishing along the divisor.
We write $\rho$ for the map defined by $D \mapsto \nu_D$ and note that $\rho$ need not be injective.

To recap, given a spherical homogeneous space $G/H$, we associate a vector space $\mathcal{N}_\mathbb{Q}$ containing the valuation cone $\mathcal{V}$, the (finite) palette $\mathcal{D}$, and a map $\rho: \mathcal{D} \rightarrow \mathcal{N}_\mathbb{Q}$.
\begin{definition}
Let $\sigma \subseteq \mathcal{N}_\mathbb{Q}$ be a cone and $\mathcal{F} \subseteq \mathcal{D}$. We call the pair $(\sigma,\mathcal{F})$ a \emph{colored cone} if the following properties are satisfied:
\begin{enumerate}
\item $\sigma$ is generated by $\rho(\mathcal{F})$ and finitely many elements of $\mathcal{V}$;
\item $\text{int}(\sigma) \cap \mathcal{V} \neq \emptyset$.
\end{enumerate}
We say that $(\sigma,\mathcal{F})$ is \emph{strictly convex} if in addition $\sigma$ is a strictly convex cone and $0 \notin \rho(\mathcal{F})$.
\end{definition}
\begin{definition}
A colored cone $(\tau,\mathcal{F}')$ is a  \emph{(colored) face} of a colored cone $(\sigma,\mathcal{F})$ if $\tau$ is a face of $\sigma$ satisfying $\text{int}(\tau) \cap \mathcal{V} \neq \emptyset$ such that $\mathcal{F}' = \mathcal{F} \cap \rho^{-1}(\tau)$. In this case we write $(\tau,\mathcal{F}') \preceq (\sigma,\mathcal{F})$ or $\tau \preceq \sigma$ if the colors are understood.
\end{definition}
\begin{definition}
A \emph{colored fan} is a finite collection $\Sigma$ of colored cones such that the following hold:
\begin{enumerate}
\item If $(\sigma,\mathcal{F}) \in \Sigma$ is a colored cone and $(\tau,\mathcal{F}')$ is a face of $(\sigma,\mathcal{F})$, then $(\tau,\mathcal{F}') \in \Sigma$.
\item Every $v \in \mathcal{V}$ is in the interior of at most one colored cone in $\Sigma$.
\end{enumerate}
We say in addition that $\Sigma$ is \emph{strictly convex} if each of its colored cones is strictly convex.
\end{definition}

We emphasize that in this definition the second condition only applies to points of $\mathcal{V}$, so that colored cones of the colored fan are allowed to overlap nontrivially outside the valuation cone.  
With these definitions in hand, we can describe the colored fan associated to a $G/H$-embedding $X$.
Note that if $D$ is a $B$-stable prime divisor on $X$ that is not $G$-stable, then the intersection $D \cap G/H$ is a color of $G/H$. Conversely, the closure in $X$ of a color is a $B$-stable prime divisor $D$ on $X$ that is not $G$-stable. Thus we can identify the palette with the set of all such divisors.
For a closed $G$-orbit $\mathcal{O}$ of $X$, let $\mathcal{F} \subseteq \mathcal{D}$ consist of the $B$-stable prime divisors containing $\mathcal{O}$ that are not $G$-stable.
The colored cone $(\sigma,\mathcal{F})$ associated to $\mathcal{O}$ is spanned by $\rho(\mathcal{F})$ and the set of $G$-stable prime divisors containing $\mathcal{O}$.
Taking these colored cones over every $G$-orbit of $X$, we obtain a colored fan.
If a $G/H$-embedding $X$ has a single closed $G$-orbit and hence a single maximal colored cone, we call $X$ \emph{simple}.
A $G/H$-embedding consists of some finite number of simple embeddings glued together along $G$-orbits; this structure is reflected in the polyhedral geometry of the colored fan.
\begin{theorem}
[\cite{LV} Prop. 8.10, \cite{Kn} Thm. 3.3]
There is a bijection between simple $G/H$-embeddings and strictly convex colored cones in $\mathcal{N}_\mathbb{Q}$, and there is a bijection between $G/H$-embeddings and strictly convex colored fans in $\mathcal{N}_\mathbb{Q}$.
\end{theorem}
\begin{example}\label{EX}
Let $G = \Sl_2$ with $B$ the subgroup of upper triangular matrices and
\[
H = \set{M \in G \mid M \text{ is upper triangular with 1's on the diagonal}}.
\]
Then $G/H = \mathbb{A}^2 \setminus \set{0}$ where the action of $G$ is given by matrix multiplication of a column vector $[x \; y]^T$.
Every character $B \rightarrow k^\ast$ is of the form
\[
\chi_n: \left( \begin{array}{cc}
a & b \\
0 & a^{-1}
\end{array} \right)
\mapsto a^n
\]
for some $n \in \mathbb{Z}$, so $\mathcal{X} \cong \mathbb{Z}$.
Under the prescribed action of $G$, we can see that $k(G/H)^{(B)}/k^* = \set{y^n \mid n \in \mathbb{Z}}$ and that the character associated to $y^n$ is $\chi_n$.
It follows that $\mathcal{M} \cong \mathbb{Z}$ and hence $\mathcal{N}_\mathbb{Q} \cong \mathbb{Q}$.

We now turn to the valuation cone $\mathcal{V}$.
Consider the following two valuations of $k(G/H)$, which are $G$-invariant:
\[
\frac{f}{g} \mapsto \mindeg{f} - \mindeg{g} \qquad \qquad \frac{f}{g} \mapsto \deg{g} - \deg{f}
\]
Here, mindeg denotes the minimum degree of a monomial in a polynomial in $k[x,y]$.
After restricting to $k(G/H)^{(B)}$, we see that the valuation on the left corresponds to sending $y^m \mapsto m$ (i.e. $\chi_1^\ast$)  
and the one on the right to $y^m \mapsto -m$ (i.e. $\chi_{-1}^\ast$).
Thus, positive multiples of these valuations induce every possible element of $\mathcal{N}_\mathbb{Q} = \text{Hom}(\mathcal{M}, \mathbb{Q})$ and so $\mathcal{V} = \mathcal{N}_\mathbb{Q}$.

The only closed $B$-orbit contained in $G/H$ is the divisor $D := V(y)$, so the palette $\mathcal{D}$ consists solely of $D$.
This means that an embedding of $G/H$ can have at most one color, corresponding to the divisor where $y$ vanishes.
This divisor gives the ray spanned by $\chi_1^\ast$.

We'll finish by explicitly computing the fan associated to $\mathbb{P}^2$ with homogeneous coordinates $W$, $X$, and $Y$. 
We can realize $\mathbb{P}^2$ as an embedding of $\mathbb{A}^2 \setminus \set{0}$ via $[x \; y]^T \mapsto [1 : x : y]$.
There are three $G$-orbits in $\mathbb{P}^2$:
\begin{align*}
\mathbb{A}^2 \setminus \set{0} & := \set{[1 : x : y] \mid x,y \in k \text{ not both zero}} \\
V(W) & := \set{[0 : x : y] \mid x,y \in k \text{ not both zero}} \\
O & := \set{[1 : 0 : 0]}.
\end{align*}
The latter two orbits are closed, so our fan will have two maximal cones. 
The orbit $V(W)$ is itself a B-stable prime divisor.
This divisor is $G$-stable, so we will have a cone without color.
The function $y$ in $k(\mathbb{A}^2 \setminus \set{0})$ can be written as $Y/W$ on $\mathbb{P}^2$.
On $V(W)$, $Y/W$ has a pole of order $1$, so the cone associated to this orbit is the ray spanned by $\chi_{-1}^\ast$.
The other closed orbit $O$ is contained in one $B$-stable prime divisor as well: $V(Y)$.
This divisor is not $G$-stable, so the corresponding ray will have color.
Clearly $y$ vanishes with order 1 on $V(Y)$, so this will give the cone spanned by $\chi_1^\ast$.
This example is drawn in Table \ref{table} along with the other colored fans of $\mathbb{A}^2 \setminus \set{0}$.
This table also appears in \cite{Vo} except for the final column, which will be explained in Example \ref{A2}. Note how the color is indicated by a bullseye.
\begin{table}[h]
\centering
\begin{tabular}{| c | c | c | c |}
\hline
Variety & Closed $G$-orbits & Colored Fan & Tropicalization \\
\hline
$\mathbb{A}^2 \setminus \set{0}$ & $\mathbb{A}^2 \setminus \set{0}$ & \begin{tikzpicture}
\draw[->,white] (0,0)--(1,0);
\draw[->,white] (0,0)--(-1,0);
\draw[fill] (0,0) circle(.05);
\end{tikzpicture} &
\begin{tikzpicture}
\draw[white,fill = white] (1,0) circle(.05);
\draw[white,fill = white] (-1,0) circle(.05);
\draw (-1,0)--(1,0);
\end{tikzpicture} \\

$\mathbb{A}^2$ & $\set{0}$ & \begin{tikzpicture}
\draw[red,->] (0,0)--(1,0);
\draw[->,white] (0,0)--(-1,0);
\draw[fill] (0,0) circle(.05);
\draw[red, fill= white] (.5,0) circle(.1);
\draw[red,fill= red] (.5,0) circle(.05);
\end{tikzpicture} &
\begin{tikzpicture}
\draw[white,fill = white] (-1,0) circle(.05);
\draw (-1,0)--(1,0);
\draw[red,fill = white] (1,0) circle(.1);
\draw[red,fill = red] (1,0) circle(.05);
\end{tikzpicture} \\

$\Bl_0(\mathbb{A}^2)$ & $E$ & \begin{tikzpicture}
\draw[->] (0,0)--(1,0);
\draw[->,white] (0,0)--(-1,0);
\draw[fill] (0,0) circle(.05);
\end{tikzpicture} &
\begin{tikzpicture}
\draw[white] (1,0) circle(.1);
\draw[fill] (1,0) circle(.05);
\draw[white,fill = white] (-1,0) circle(.05);
\draw (-1,0)--(1,0);
\end{tikzpicture} \\

$\mathbb{P}^2 \setminus \set{0}$ & $V(W)$ & \begin{tikzpicture}
\draw[->,white] (0,0)--(1,0);
\draw[->] (0,0)--(-1,0);
\draw[fill] (0,0) circle(.05);
\end{tikzpicture} & 
\begin{tikzpicture}
\draw[white] (-1,0) circle(.1);
\draw[fill] (-1,0) circle(.05);
\draw[white,fill = white] (1,0) circle(.05);
\draw (-1,0)--(1,0);
\end{tikzpicture} \\

$\mathbb{P}^2$ & $V(W)$, $\set{0}$ & \begin{tikzpicture}
\draw[red,->] (0,0)--(1,0);
\draw[->] (0,0)--(-1,0);
\draw[fill] (0,0) circle(.05);
\draw[red, fill= white] (.5,0) circle(.1);
\draw[red,fill= red] (.5,0) circle(.05);
\end{tikzpicture} & 
\begin{tikzpicture}
\draw[fill] (-1,0) circle(.05);
\draw (-1,0)--(1,0);
\draw[red,fill = white] (1,0) circle(.1);
\draw[red,fill = red] (1,0) circle(.05);
\end{tikzpicture} \\

$\Bl_0(\mathbb{P}^2)$ & $V(W)$, $E$ & \begin{tikzpicture}
\draw[->] (0,0)--(1,0);
\draw[->] (0,0)--(-1,0);
\draw[fill] (0,0) circle(.05);
\end{tikzpicture} & 
\begin{tikzpicture}
\draw[white] (1,0) circle(.1);
\draw[fill] (1,0) circle(.05);
\draw[fill] (-1,0) circle(.05);
\draw (-1,0)--(1,0);
\end{tikzpicture} \\
\hline
\end{tabular}
\caption{Colored fans and colored tropicalizations associated to $\mathbb{A}^2 \setminus \set{0}$. The $E$ denotes the exceptional divisor of the blowup.}
\label{table}
\end{table}
\end{example}

\section{Tropicalizing Homogeneous Spaces}\label{HomSpaces}

In his thesis \cite{Vo}, Tassos Vogiannou defines the tropicalization of a subvariety of the spherical homogeneous space $G/H$, extending the well-known theory of subvarieties of a torus.
We outline his construction here; more details and examples can be found in his thesis.
Suppose $G/H$ is a spherical homogeneous space over $k$, let $K := k((t))$ denote the field of Laurent series, and let $\overline{K} := \bigcup_{n \in \mathbb{N}} k((t^{1/n}))$ denote the field of Puiseaux series.
We use the valuation $\nu: \overline{K} \rightarrow \mathbb{Q}$ that gives the lowest power of $t$ appearing with nonzero coefficient.
Note that this restricts naturally to $K$ and is trivial on $k$.

We will define a map $G/H(K) \rightarrow \mathcal{N}_\mathbb{Q}$.
Let $\gamma: \spec{K} \rightarrow G/H$ be a $K$-point of $G/H$.
We will define a $G$-invariant discrete valuation $\nu_\gamma$ on $k(G/H)^\ast$ associated to $\gamma$.
To do this, we need to describe how $\nu_\gamma$ acts on rational functions, so let $f \in k(G/H)^\ast$ be arbitrary.
The domain of $f$ may not contain the image of $\gamma$, but we may find $g \in G$ such that the image of $\gamma$ is in the domain of $gf$.
There is a pullback map $\gamma^\ast: k(G/H) \rightarrow K$ given by evaluation at $\gamma$, so we consider $\gamma^\ast(gf) \in K$.
Then we write $\nu_\gamma(f) = \nu(\gamma^\ast(gf))$.
This is not a priori well-defined since it may depend on $g$.
To overcome this, we take $g$ so that $\nu(\gamma^\ast(gf))$ is minimized; this minimum is achieved on an open set of $G$ and we call such $g$ \emph{sufficiently general}.

Thus we have a map $G/H(K) \rightarrow \set{G\text{-invariant discrete valuations on } k(G/H)^\ast}$ given by $\gamma \mapsto \nu_\gamma$.
As discussed, $G$-invariant discrete valuations on $k(G/H)^\ast$ determine elements of $\mathcal{V}$, so this is really a map $G/H(K) \rightarrow \mathcal{V}$.
Further, we can extend this map so it is defined over $G/H(\overline{K})$.
Indeed, suppose $\gamma: \spec{\overline{K}} \rightarrow G/H$ is a $\overline{K}$-point.
This induces a homomorphism of $k$-algebras $\gamma^\ast: A \rightarrow \overline{K}$ since the image of $\gamma$ must lie in some open affine $\spec{A} \subseteq X$.
Since $G/H$ is of finite type, $A$ is finitely-generated as a $k$-algebra, and so it follows that $\gamma^\ast$ factors through $k((t^{1/n}))$ for some sufficiently large $n$.
Thus, $\gamma$ factors as $\spec{\overline{K}} \rightarrow \spec{k((t^{1/n}))} \rightarrow G/H$.
We can think of $\spec{k((t^{1/n}))}$ as the spectrum of Laurent polynomials in an indeterminate variable $t^{1/n}$.
This morphism induces a valuation by the work above; dividing this valuation by $n$ gives a valuation $\nu_\gamma$, which we associate to $\gamma$.
This extension in fact gives a surjection $\val: G/H(\overline{K}) \twoheadrightarrow \mathcal{V}$, which allows us to finally define the tropicalization of a subvariety of a homogeneous space.
\begin{definition}
If $Y \subseteq G/H$ is a subvariety, the \emph{tropicalization} of $Y$ is $\trop_G(Y) := \val\left(Y\left(\overline{K}\right)\right)$.
\end{definition}

In particular, note that $\trop_G(G/H) = \mathcal{V}(G/H)$, a fact we will use in \textsection \ref{construction}.

\section{The Construction}\label{construction}

Again let $G/H$ be a spherical homogeneous space and let $\mathcal{N}_\mathbb{Q}$ and $\mathcal{V} \subseteq \mathcal{N}_\mathbb{Q}$ be the associated vector space and valuation cone.
Let $X$ be a simple $G/H$-embedding $X$ with maximal colored cone $\left(\sigma, \mathcal{F} \right)$.
We will show how to tropicalize $X$ and then see how the tropicalization of a general embedding can be obtained by tropicalizing simple embeddings and gluing.
Then we will describe how to tropicalize a subvariety.

Each colored face $\tau$ of $\sigma$ corresponds to a $G$-orbit $\mathcal{O}_\tau$ of $X$, and Corollary 2.2 in \cite{Kn} says that each orbit is a spherical homogeneous $G$-variety with an open orbit of the same Borel subgroup $B$.
As a spherical homogeneous space, an orbit $\mathcal{O}_\tau$ associated to $\tau$ has a valuation cone $\mathcal{V}_\tau := \mathcal{V}\left(\mathcal{O}_\tau \right)$ that lies in a $\mathbb{Q}$-vector space $\mathcal{N}_\mathbb{Q}\left( \mathcal{O}_\tau \right)$.
Then as a set, we define $\trop_G(X) := \bigsqcup_{\tau \preceq \sigma} \mathcal{V}_\tau$.
This is similar in spirit to the construction of \cite{Ka} and \cite{Pay}; we break $X$ up into orbits and tropicalize each of them separately. It only remains to define a topology on this set.

Let $\overline{\mathbb{Q}} := \mathbb{Q} \cup \set{\infty}$.
We write 
\[
\overline{\mathcal{N}}(\sigma) := \Hom\left( \sigma^\vee \cap \mathcal{M}, \overline{\mathbb{Q}} \right),
\] 
where the homomorphisms in the set on the right are semigroup homomorphisms.
We will show that as sets $\bigsqcup_{\tau \preceq \sigma} \mathcal{N}_\mathbb{Q}(\mathcal{O}_\tau) \subseteq \overline{\mathcal{N}}_\mathbb{Q}$.

We fix our attention on a colored face $\tau \preceq \sigma$.
There is a copy of $\mathcal{N}_\mathbb{Q}\left( \mathcal{O}_\tau \right)$ in $\overline{\mathcal{N}}(\sigma)$ given by considering those semigroup homomorphisms $\varphi: \sigma^\vee \cap \mathcal{M} \rightarrow \overline{\mathbb{Q}}$ for which $\varphi^{-1}(\mathbb{Q}) = \tau^\perp \cap \left(\sigma^\vee \cap \mathcal{M} \right)$.
More explicitly,
\[
\Hom{\left(\tau^\perp \cap \left(\sigma^\vee \cap \mathcal{M} \right),\mathbb{Q} \right)} \cong \Hom{\left(\tau^\perp \cap \mathcal{M},\mathbb{Q} \right)} \cong \mathcal{N}_\mathbb{Q}(\mathcal{O}_\tau),
\]
which we see as follows.
We have that $\tau^\perp \cap \mathcal{M}$ consists of those functions in $\mathcal{M} = k(G/H)^{(B)}/k^\ast$ that do not have zeroes or poles along the orbit $\mathcal{O}_\tau$.
These are precisely the functions in $\mathcal{M}$ that can be restricted to $B$ semi-invariant rational functions on $\mathcal{O}_\tau$.
Restriction thus gives a map $\tau^\perp \cap \left( \sigma^\vee \cap \mathcal{M} \right) \rightarrow \mathcal{M}\left( \mathcal{O}_\tau \right)$.
Theorem 6.3 of \cite{Kn} shows that this map is an isomorphism, so after dualizing we have $\Hom{\left(\tau^\perp \cap \left( \sigma^\vee \cap \mathcal{M} \right),\mathbb{Q} \right)} \cong \mathcal{N}_\mathbb{Q}\left( \mathcal{O}_\tau \right)$.
Further, homomorphisms in $\Hom{\left( \tau^\perp \cap \left( \sigma^\vee \cap \mathcal{M} \right),\mathbb{Q} \right)}$ extend uniquely to homomorphisms in $\overline{\mathcal{N}}(\sigma)$ by sending every character outside $\tau^\perp \cap \mathcal{M}$ to $\infty$.

In the toric case, it now follows that $\bigsqcup_{\tau \preceq \sigma} \mathcal{N}_\mathbb{Q}\left( \mathcal{O}_\tau \right)$ is in bijective correspondence with $\overline{\mathcal{N}}(\sigma)$.
For a general spherical variety, this is not necessarily true. 
This is because a colored cone may contain a subcone which is a face in the sense of polyhedral geometry but which lies outside the valuation cone and therefore does not correspond to an orbit in the spherical variety. 
To address this, we write $\overline{\mathcal{N}}_\mathcal{V}(\sigma)$
to denote the set of homomorphisms $\varphi \in \overline{\mathcal{N}}(\sigma)$ such that $\varphi^{-1}(\mathbb{Q}) = \tau^\perp \cap \left(\sigma^\vee \cap \mathcal{M} \right)$ for some colored face $\tau \preceq \sigma$.

Now every homomorphism $\varphi \in \overline{\mathcal{N}}_\mathcal{V}(\sigma)$ is realized as an extension of a homomorphism $\tau^\perp \cap \left(\sigma^\vee \cap \mathcal{M} \right) \rightarrow \mathbb{Q}$ for some unique $\tau \preceq \sigma$ where $\varphi^{-1}(\mathbb{Q}) =  \tau^\perp \cap \left(\sigma^\vee \cap \mathcal{M} \right)$.
Thus we have a bijective correspondence:
\[
\overline{\mathcal{N}}_\mathcal{V}(\sigma) \leftrightarrow \bigsqcup_{\tau \preceq \sigma} \mathcal{N}_\mathbb{Q}\left(\mathcal{O}_\tau \right).
\]
Placing the topology on $\overline{\mathcal{N}}_\mathcal{V}(\sigma)$ inherited from $\overline{\mathbb{Q}}^{\sigma^\vee \cap \mathcal{M}}$, we obtain a topology on $\bigsqcup_{\tau \preceq \sigma} \mathcal{N}_\mathbb{Q}\left( \mathcal{O}_\tau \right)$.
Under this topology, $\bigsqcup_{\tau \preceq \sigma} \mathcal{N}_\mathbb{Q}\left( \mathcal{O}_\tau \right)$ is isomorphic to a subspace of $\overline{\mathbb{Q}}^{m}$, where $m$ is the rank of $\mathcal{N}$.

We still must see how $\bigsqcup_{\tau \preceq \sigma} \mathcal{V}_\tau$ lies in $\overline{\mathcal{N}}_\mathcal{V}(\sigma)= \bigsqcup_{\tau \preceq \sigma} \mathcal{N}_\mathbb{Q}\left(\mathcal{O}_\tau \right)$ under this topology.
Any homomorphism in $\Hom{(\tau^\perp \cap \mathcal{M},\mathbb{Q})}$ that is induced by a $G$-invariant valuation is the restriction of a $G$-variant valuation in $\mathcal{V}(G/H) \subseteq \Hom{(\mathcal{M},\mathbb{Q})}$.
Indeed, by Corollary 1.5 in \cite{Kn}, any $G$-invariant valuation on $k\left(\mathcal{O}_\tau \right)$ can be lifted to a $G$-invariant valuation on $k(G)$, which can then be restricted to $k(G/H)$.
Thus, the valuation cone $\mathcal{V}_\tau$ is the image of the valuation cone $\mathcal{V}$ under the map $\Hom{(\mathcal{M},\mathbb{Q})} \rightarrow \Hom{\left(\tau^\perp \cap \mathcal{M},\mathbb{Q}\right)}$ induced by the inclusion $\tau^\perp \cap \mathcal{M} \hookrightarrow \mathcal{M}$.
Thus $\trop_G(X) = \bigsqcup_{\tau \preceq \sigma} \mathcal{V}_\tau$ inherits the subspace topology from $\overline{\mathcal{N}}_\mathcal{V}(\sigma)$.

To tropicalize a non-simple spherical embedding, we tropicalize the simple embeddings corresponding to each of its maximal cones and then glue these together along the tropicalizations of their shared orbits. If $Y \subseteq G/H$ is a closed subvariety and $\overline{Y} \subseteq X$ its closure in $X$, then we define
\[
\trop_G\left(\overline{Y}\right) := \bigsqcup_{\tau \preceq \sigma} \trop_G \left(\overline{Y} \cap \mathcal{O}_\tau \right) \subseteq \bigsqcup_{\tau \preceq \sigma} \mathcal{V}_\tau = \trop_G(X).
\]
Here, $\trop_G\left(\overline{Y} \cap \mathcal{O}_\tau \right)$ denotes the tropicalization as a subvariety of a homogeneous space and we give $\trop_G\left( \overline{Y} \right)$ the subspace topology inherited from $\trop_G(X)$.

Our construction as discussed thus far only recognizes the polyhedral structure of the colored fan but ignores whether or not that fan has colors.
We need this information to completely classify all spherical embeddings, so it seems useful to remember the presence of colors when we tropicalize.
We address this as follows.
In a colored fan, we can think of our palette of colors as a collection of points in $\mathcal{N}_\mathbb{Q}$ corresponding to $B$-stable prime divisors.
If no colors appear in our fan, we call the spherical variety \emph{toroidal}.
If a color appears, it will lie in some number of colored cones, which is to say the associated prime divisor contains the orbits corresponding to those cones.
Each such orbit gives a cone in the stratification of the tropicalization; we simply record whether the color appears in the colored fan by labeling its associated valuation cone with that color.
We show in Example \ref{A2} two different spherical embeddings of the same homogeneous space that have different colored tropicalizations; if color is ignored the tropicalizations become the same.



\section{Examples}\label{examples}

\begin{example}\label{toricex}
In the toric case, $G = T^n$ is a torus of dimension $n$, $H$ is trivial, and $B = G$, so there are no colors.
The $B$ semi-invariant rational functions are precisely the monomials in the variables $\set{x_1,\ldots,x_n}$, so $\mathcal{N}_\mathbb{Q} \cong \mathbb{Q}^n$, spanned by cocharacters $\chi_i^*$ defined as follows:
\[
\chi^*_i(x_j) = \left\lbrace \begin{array}{ll}
1 & i = j \\
0 & \text{otherwise.}
\end{array}
\right.
\]
We can also see that the valuation cone $\mathcal{V}$ is all of $\mathcal{N}_\mathbb{Q}$.
Indeed, consider the valuations
\[
\frac{f}{g} \mapsto \text{mindeg}_i(f) - \text{mindeg}_i(g) \qquad \qquad \frac{f}{g} \mapsto \text{deg}_i(g) - \text{deg}_i(f),
\]
where $\text{deg}_i$ and $\text{mindeg}_i$ respectively denote the degree and minimum degree in $x_i$.
These valuations are $G$-invariant and induce the cocharacters $\chi_i^*$ and $-\chi_i^\ast$, so $\mathcal{V} = \mathcal{N}_\mathbb{Q}$.

Having realized the torus as a spherical homogeneous space, the colored fan associated to a toric variety when viewed as a spherical embedding is the same as the fan coming from the theory of toric geometry.
Our definition of the tropicalization only relies on the polyhedral structure of this fan and is identical to the process described in \cite{Ka} and \cite{Pay}.
\end{example}
\begin{example}\label{flag}
A flag variety is a spherical homogeneous space $G/P$ where $P$ is a parabolic subgroup, one which contains a Borel subgroup.
Such a homogenous space has a trivial valuation cone, so the tropicalization of a flag variety is a point under this theory.
\end{example}
\begin{example}\label{A2}
We return now to the example and notation of $\mathbb{A}^2\setminus \set{0}$ discussed in Example \ref{EX}.
In $\Bl_0(\mathbb{P}^2)$, there are three $G$-orbits: $G/H$, $V(W)$, and the exceptional divisor $E$. We have already seen that the valuation cone of $G/H$ is a copy of $\mathbb{Q}$, so we move on to $V(W)$ and $E$.
Both of these are copies of $\mathbb{P}^1$ and the action of $G = \Sl_2$ on both of them is given by matrix multiplication.
In this case $G/B \cong \mathbb{P}^1$, so the closed orbits are flag varieties and our discussion in example \ref{flag} tells us the tropicalizations are trivial.
Let us show this explicitly.
The rational functions $k(G/B) = k(\mathbb{P}^1)$ are quotients of two homogeneous polynomials of the same degree in the variables $X$ and $Y$.
The action of the Borel subgroup $B$ on these functions is the same as it is on $k(G/H)$, so the only $B$ semi-invariant rational functions on $k(G/B)$ are powers of $Y$.
The only power of $Y$ in $k(\mathbb{P}^1)$ is the constant function, so $k(\mathbb{P}^1)^{(B)}$ is trivial and hence so are the associated $\mathcal{M}$, $\mathcal{N}_\mathbb{Q}$, and $\mathcal{V}$.

Thus, $\bigsqcup_\tau \mathcal{V}\left( \mathcal{O}_\tau \right)$ in this case consists of a copy of $\mathbb{Q}$ and two points.
The two points attach to $\mathbb{Q}$ by thinking of them as $\infty$ and $-\infty$.
We can think of $\Bl_0(\mathbb{P}^2)$ as the two simple spherical varieties $\Bl_0(\mathbb{A}^2)$ and $\mathbb{P}^2 \setminus \set{0}$ glued together along $G/H = \mathbb{A}^2 \setminus \set{0}$.
In $\Bl_0(\mathbb{A}^2)$, we add in limit points over 0, which correspond to extended valuations taking $y \in k(G/H)^{(B)}$ to $\infty$, giving a copy of $\overline{\mathbb{Q}}$.
In $\mathbb{P}^2 \setminus \set{0}$, we add in limit points at infinity, which similarly correspond to extended valuations and we again get $\overline{\mathbb{Q}}$.
We finally glue these copies of $\overline{\mathbb{Q}}$ along their shared copy of $\mathbb{Q}$ by identifying a number in one copy of $\mathbb{Q}$ with its reciprocal in the other copy.
This is illustrated in Figure \ref{BLP2}. The gluing is reminiscent of the tropicalization of $\mathbb{P}^1$ viewed as a toric variety, as described in \textsection 6.2 of \cite{MS}.

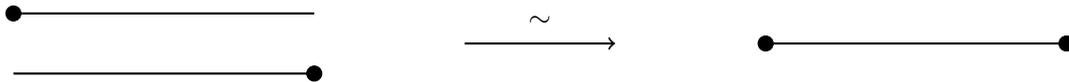
\begin{figure}[!h]
\begin{center}
\begin{tikzpicture}[scale = 2]
\draw[thick] (-4,-.2)--(-6,-.2);
\draw[thick] (-4,.2)--(-6,.2);
\draw[fill] (-4,-.2) circle(.05);
\draw[fill] (-6,.2) circle(.05);

\draw (-2.5,.15) node[]{$\sim$};
\draw[thick,->] (-3,0)--(-2,0);

\draw[thick] (-1,0)--(1,0);
\draw[fill] (-1,0) circle(.05);
\draw[fill] (1,0) circle(.05);
\end{tikzpicture}
\caption{The colored tropicalization of $\Bl_0(\mathbb{P}^2)$}
\label{BLP2}
\end{center}
\end{figure}
If instead we consider the $G/H$-embedding $\mathbb{P}^2$, the simple spherical varieties are $\mathbb{P}^2 \setminus \set{0}$ and $\mathbb{A}^2$. The former can be tropicalized as before, but $\mathbb{A}^2$ has a $G$-fixed point $[1:0:0]$ whose associated cone has color.
The point has a trivial valuation cone, just like the effective divisor in $\Bl_0(\mathbb{P}^2)$.
The gluing operation works the same as with $\Bl_0(\mathbb{P}^2)$, so topologically we again obtain a line segment.
This is shown in Figure \ref{ColP2}; the colored point associated to $[1:0:0]$ is on the right.  
\begin{figure}[!h]
\begin{center}
\begin{tikzpicture}[scale = 2]
\draw[thick] (-1,0)--(1,0);
\draw[fill] (-1,0) circle(.05);
\draw[red,fill=white] (1,0) circle(.1);
\draw[red,fill= red] (1,0) circle(.05);
\end{tikzpicture}
\caption{The colored tropicalization of $\mathbb{P}^2$}
\label{ColP2}
\end{center}
\end{figure}
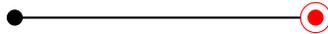 
The other embeddings of $\mathbb{A}^2 \setminus \set{0}$ are obtained from $\Bl_0\left( \mathbb{P}^2 \right)$ or $\mathbb{P}^2$ by omitting certain orbits.
Thus, their tropicalizations omit the corresponding pieces, and we obtain the other results depicted in Table \ref{table}.
\end{example}
\begin{example}
This example extends \textsection 5.3 of \cite{Vo}. Our group is $G = \Gl_2 \times \Gl_2$ and the subgroup $H$ is the diagonal, so that $G/H \cong \Gl_2$. the action of $G$ on $G/H$ is given by $(g,h) \cdot X = gXh^{-1}$ and the Borel subgroup is 
\[
B = \set{(U,L) \mid U \text{ is upper triangular and } L \text{ is lower triangular}}.
\]
The Borel subgroup has an open orbit $\set{(x_{ij}) \in \Gl_2 \mid x_{22} \neq 0}$, so this is a spherical homogeneous space.
We will embed $G/H$ into $\Bl_0(\mathbb{A}^4)$ by sending $X = (x_{ij})$ to its image in the principal open subset $D(x_{11}x_{22} - x_{12}x_{21}) \subset \mathbb{A}^4 \setminus \set{0}$.
Viewing $\Bl_0(\mathbb{A}^4)$ as a subvariety of $\mathbb{A}^4 \times \mathbb{P}^3$, we give $\mathbb{P}^3$ the coordinates $y_{ij}$ and denote elements by $((x_{ij}),[y_{ij}])$, so that the blowup is cut out by the equations $x_{ij}y_{k\ell} = x_{k\ell}y_{ij}$.
The action of $G$ on $\Bl_0(\mathbb{A}^4)$ is then given by matrix multiplication in both components:
\[
(g,h) \cdot ((x_{ij}),[y_{ij}]) = (g(x_{ij})h^{-1},g[y_{ij}]h^{-1}), \quad (g,h) \in G, ((x_{ij}),[y_{ij}]) \in \Bl_0(\mathbb{A}^4).
\]
In this case, the lattice of $B$ semi-invariant rational functions $\mathcal{M} = k(G/H)^{(B)}/k^*$ on $G/H$ is spanned by $f_1 := (x_{11}x_{22} - x_{12}x_{21})/x_{22}$ and $f_2 := x_{22}$. We choose these particular generators because the associated cocharacters are cleaner for computations.
The palette $\mathcal{D}$ in this case consists of one $B$-stable prime divisor: $V(x_{22})$.
The vector space $\mathcal{N}_\mathbb{Q}$ of cocharacters is two-dimensional, spanned by $\chi_1^*$ and $\chi_2^*$ defined as follows:
\[
\chi_{i}^*(f_j) = \left\lbrace \begin{array}{ll}
1 & \text{if $i = j$} \\
0 & \text{otherwise}
\end{array}
\right..
\]
The valuation cone $\mathcal{V}$ associated to $G/H$ is $\set{\alpha_1\chi_1^* + \alpha_2\chi_2^* \mid \alpha_1 \geq \alpha_2}$. The valuation cone and palette of $G/H$ are shown in Figure \ref{valcone}.

Let us now determine the colored fan associated to $\Bl_0\left(\mathbb{A}^4 \right)$ as a $\Gl_2$-embedding. Under the prescribed action of $G$, there are four $G$-orbits: 
\begin{align*}
\Gl_2 & := \set{((x_{ij}),[y_{ij}]) \mid (x_{ij}) \in \Gl_2} \\
R_1 & := \set{((x_{ij}),[y_{ij}]) \mid (x_{ij}) \text{ has  rank } 1} \\
\mathbb{P}(\Gl_2) & :=  \set{(0,[y_{ij}]) \mid (y_{ij}) \in \Gl_2} \\
\mathbb{P}(R_1) & := \set{(0,[y_{ij}]) \mid (y_{ij}) \text{ has rank } 1}
\end{align*}
Only one of these orbits is closed: $\mathbb{P}(R_1)$, so we will have one maximal colored cone.

There are three $B$-stable prime divisors of $\Bl_0\left( \mathbb{A}^4 \right)$: $V(y_{11}y_{22} - y_{12}y_{21}) = V(\text{det}(y_{ij}))$, the exceptional divisor $E$, and the color $V(x_{22})$.
The only $B$-stable prime divisors containing $\mathbb{P}(R_1)$ are $E$ and $V(\text{det}(y_{ij}))$.
Both of these are also $G$-stable, so our fan will have no colors.
Along $E$, $f_2$ clearly vanishes with order 1 and $f_1$ can be written in the form $f_1 = f_2 \cdot (y_{11}y_{22} - y_{12}y_{21})/y_{22}^2$, so it also vanishes with order 1 along $E$. Thus we obtain a ray $\sigma_{1,1}$ in the direction $(1,1)$.
Along $V(\text{det}(y_{ij}))$, $f_1$ vanishes with order 1 and $f_2$ doesn't vanish, so we obtain a ray $\sigma_{1,0}$ in the direction $(1,0)$.
Together $\sigma_{1,0}$ and $\sigma_{1,1}$ span a single two-dimensional cone $\sigma$.
Figure \ref{blowupA4} exhibits the colored cone associated to $\Bl_0(\mathbb{A}^4)$.
\begin{figure}[!h]
\begin{floatrow}
\ffigbox{\begin{tikzpicture}
\draw[fill,gray!25] (-2,-2)--(2,-2)--(2,2);
\draw[thick,->] (0,0)--(2,2);
\draw[fill] (0,0) circle(.05);
\draw[thick,->] (0,0)--(-2,-2);
\draw[red,fill=red] (-1,1) circle(.05);
\draw[red] (-1,1) circle(.1);

\draw[->,dotted] (0,0)--(2,0);
\draw[->,dotted] (0,0)--(-2,0);
\draw[->,dotted] (0,0)--(0,2);
\draw[->,dotted] (0,0)--(0,-2);

\draw (2.3,0) node {$\alpha_1$};
\draw (0,2.15) node {$\alpha_2$};
\end{tikzpicture}
\caption{The valuation cone and palette of $\Gl_2$}
\label{valcone}}

\ffigbox{\begin{tikzpicture}
\draw[fill,gray!25] (0,0)--(2,2)--(2,0);
\draw[fill] (0,0) circle(.05);

\draw[very thick,->] (0,0)--(2,0);
\draw[very thick,->] (0,0)--(2,2);

\draw[->,dotted] (0,0)--(2,0);
\draw[->,dotted] (0,0)--(-2,0);
\draw[->,dotted] (0,0)--(0,2);
\draw[->,dotted] (0,0)--(0,-2);

\draw (2.3,0) node {$\alpha_1$};
\draw (0,2.15) node {$\alpha_2$};
\end{tikzpicture}
\caption{The colored cone of $\Bl_0(\mathbb{A}^4)$.}
\label{blowupA4}}

\end{floatrow}
\end{figure}

Now we apply our construction. 
We start with $\bigsqcup \mathcal{N}_\mathbb{Q}(\mathcal{O}_i)$, where the disjoint union is over four separate colored cones corresponding to the four $G$-orbits in $\Bl_0(\mathbb{A}^4)$.
There is one colored cone of dimension zero $(\Gl_2)$, two of dimension one $(R_1 \text{ and } \mathbb{P}(\Gl_2))$, and one of dimension two $(\mathbb{P}(R_1))$, so $\bigsqcup \mathcal{N}_\mathbb{Q}(\mathcal{O}_i)$ consists of one copy of $\mathbb{Q}^2$, two copies of $\mathbb{Q}^1$, and one zero-dimensional vector space.
They are shown in Figure \ref{QA4}. The vertical line corresponds to the orbit $R_1$ and the slanted line to $\mathbb{P}(GL_2)$.
\begin{figure}[!h]
\begin{tikzpicture}
\draw[fill] (0,0) circle(.05);
\draw[fill,gray!25] (-2,1)--(-2,-3)--(-6,-3)--(-6,1);

\draw[thick] (-.5,.5)--(-2.5,2.5);
\draw[very thick] (0,-.75)--(0,-3);
%
%
\end{tikzpicture}
\caption{$\bigsqcup \mathcal{N}_\mathbb{Q}(\mathcal{O}_i)$}
\label{QA4}
\end{figure}
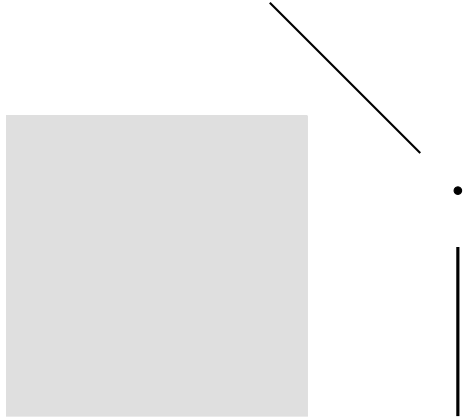

Now we consider the valuation cone $\mathcal{V}$ of each orbit.
We have already seen that the valuation cone of $\Gl_2$ is given by $\set{(\alpha_1,\alpha_2) \mid \alpha_1 \geq \alpha_2}$ and the valuation cone of $\mathbb{P}(R_1)$ is necessarily trivial.
The orbit $R_1$ is a spherical homogeneous space isomorphic to $V(x_{11}x_{22} - x_{12}x_{21}) \subset \mathbb{A}^4 \setminus \set{0}$.
The $B$ semi-invariant rational functions $k(R_1)^{(B)}$ are just the powers of $f_2$ since $f_1$ vanishes along $R_1$.
We can define valuations that consider the degree of a rational function similarly to Example \ref{EX}, so the valuation cone $\mathcal{V}(R_1)$ is a copy of $\mathbb{Q}$.
Alternatively, $\mathcal{V}(R_1)$ is the image of $\mathcal{V}(\Gl_2)$ under the projection map $\mathcal{N}_\mathbb{Q}(\Gl_2) \rightarrow\mathcal{N}_\mathbb{Q}(R_1)$.
This map is defined by taking a product $f_1^{\alpha_1}f_2^{\alpha_2} \in k(\Gl_2)^{(B)}/k^*$ to $f_2^{\alpha_2} \in k(R_1)^{(B)}/k^*$.
The image of $\mathcal{V}(\Gl_2)$ is all of $\mathcal{N}_\mathbb{Q}(R_1)$, so $\mathcal{V}(R_1) = \mathcal{N}_\mathbb{Q}(R_1)$.

Finally, the orbit $\mathbb{P}(\Gl_2)$ is $D(y_{11}y_{22} - y_{12}y_{21}) \subset \mathbb{P}^3$.
The $B$ semi-invariant rational functions $k(\mathbb{P}(\Gl_2))^{(B)}$ are spanned by $(y_{11}y_{22} - y_{12}y_{21})/y_{22}^2$ since they must have the same degree in the numerator and denominator.
Since $\nu(f_1) \geq \nu(f_2)$ for any $G$-invariant valuation $\nu$ on $\mathcal{M}(\Gl_2)$, we have $\nu((y_{11}y_{22} - y_{12}y_{21})/y_{22}^2) \geq 0$ for any $G$-invariant valuation $\nu$ on $\mathcal{M}(\mathbb{P}(\Gl_2))$.
Thus, $\mathcal{V}(\mathbb{P}(\Gl_2))$ is a ray in $\mathcal{N}_\mathbb{Q}(\mathbb{P}(\Gl_2))$.
The union of the valuation cones and their gluing is illustrated in Figure \ref{tropA4}.
In Table \ref{coltrops} we exhibit several other embeddings of $\Gl_2$ along with their associated colored cones and colored tropicalizations.
\begin{figure}[!h]
\begin{tikzpicture}
\draw[fill] (0,0) circle(.05);
\draw[fill] (-1.5,1.5) circle(.05);
\draw[fill,gray!25] (-2,1)--(-2,-3)--(-6,-3);

\draw[thick] (-6,-3)--(-2,1);
\draw[thick] (-.5,.5)--(-1.5,1.5);
\draw[very thick] (0,-.75)--(0,-3);

\draw[thick,->] (1.5,-1)--(3.5,-1);

\draw[fill,gray!25] (7,1)--(8,0)--(8,-3)--(5,-1);
\draw[fill] (8,0) circle(.05);
\draw[fill] (7,1) circle(.05);
\draw[thick] (7,1)--(5,-1);
\draw[thick] (7,1)--(8,0);
\draw[very thick] (8,-3)--(8,0);
\end{tikzpicture}
\caption{$\bigsqcup \mathcal{V}_i$ and the tropicalization of $\Bl_0(\mathbb{A}^4)$}
\label{tropA4}
\end{figure}
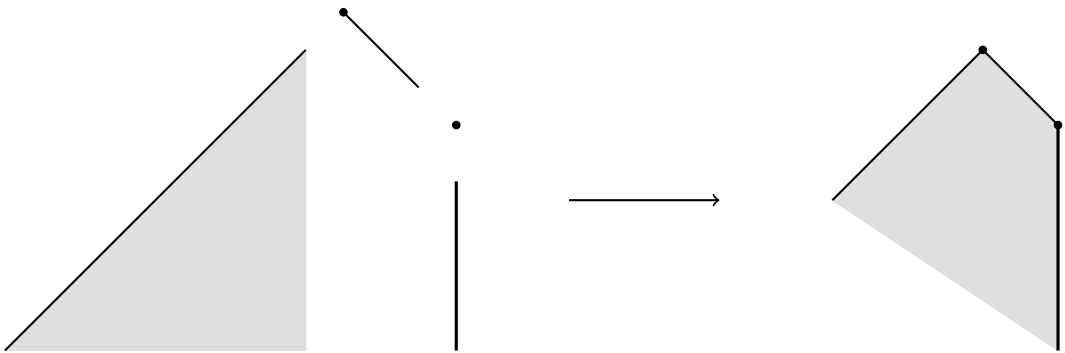

\begin{table}[h!]
\begin{tabular}{| >{\centering\arraybackslash} m{4cm} >{\centering\arraybackslash} m{4cm} >{\centering\arraybackslash} m{4cm} |}
\hline
Variety & Colored Fan & Colored Tropicalization \\
\hline
\raisebox{1.4cm}{$\Gl_2$} \rule{0pt}{2.8cm} & \begin{tikzpicture}[scale = .5]
\draw[fill] (0,0) circle(.1);

\draw[->,dotted] (0,0)--(2,0);
\draw[->,dotted] (0,0)--(-2,0);
\draw[->,dotted] (0,0)--(0,2);
\draw[->,dotted] (0,0)--(0,-2);

\end{tikzpicture} &
\begin{tikzpicture}
\draw[fill,gray!25] (0,0)--(1,1)--(1,-1);
\draw[thick] (0,0)--(1,1);
\end{tikzpicture} \\

\raisebox{1.4cm}{$\mathbb{A}^4 \setminus \set{0}$} \rule{0pt}{2.8cm} & 
\begin{tikzpicture}[scale = .5]
\draw[fill] (0,0) circle(.1);
\draw[very thick,->] (0,0)--(2,0);

\draw[->,dotted] (0,0)--(2,0);
\draw[->,dotted] (0,0)--(-2,0);
\draw[->,dotted] (0,0)--(0,2);
\draw[->,dotted] (0,0)--(0,-2);

\end{tikzpicture}
& 
\begin{tikzpicture}
\draw[fill,gray!25] (0,0)--(1,1)--(1,-1);
\draw[thick] (0,0)--(1,1);
\draw[very thick] (1,1)--(1,-1);

\draw[fill = white] (1,1) circle(.1);
\end{tikzpicture} \\

\raisebox{1.4cm}{$\mathbb{A}^4$} \rule{0pt}{2.8cm} 
& 
\begin{tikzpicture}[scale = .5]
\draw[fill,red!50] (-2,2)--(0,0)--(2,0)--(2,2);
\draw[very thick,red,->] (0,0)--(-2,2);
\draw[red,fill=white] (-1,1) circle(.2);
\draw[red, fill=red] (-1,1) circle(.1);
\draw[fill] (0,0) circle(.1);

\draw[very thick,->] (0,0)--(2,0);
\draw[->,dotted] (0,0)--(-2,0);
\draw[->,dotted] (0,0)--(0,2);
\draw[->,dotted] (0,0)--(0,-2);

\end{tikzpicture}
& 
\begin{tikzpicture}
\draw[fill,gray!25] (0,0)--(1,1)--(1,-1);
\draw[thick] (0,0)--(1,1);
\draw[very thick] (1,1)--(1,-1);

\draw[red,fill = white] (1,1) circle(.1);
\draw[red,fill = red] (1,1) circle(.05);
\end{tikzpicture} \\

\raisebox{1.4cm}{$\Bl_0(\mathbb{A}^4)$} \rule{0pt}{2.8cm} 
& 
\begin{tikzpicture}[scale = .5]
\draw[fill,gray!25] (0,0)--(2,0)--(2,2);

\draw[fill] (0,0) circle(.1);

\draw[very thick,->] (0,0)--(2,2);
\draw[very thick,->] (0,0)--(2,0);
\draw[->,dotted] (0,0)--(2,0);
\draw[->,dotted] (0,0)--(-2,0);
\draw[->,dotted] (0,0)--(0,2);
\draw[->,dotted] (0,0)--(0,-2);

\end{tikzpicture}
& 
\begin{tikzpicture}
\draw[fill,gray!25] (0,0)--(.5,.5)--(1,0)--(1,-1);
\draw[very thick] (1,0)--(1,-1);
\draw[thick] (0,0)--(.5,.5);
\draw[thick] (.5,.5)--(1,0);

\draw[fill] (.5,.5) circle(.05);
\draw[fill] (1,0) circle(.05);
\end{tikzpicture} \\

\raisebox{1.4cm}{$\mathbb{P}^4 \setminus \set{0}$} \rule{0pt}{2.8cm} 
& 
\begin{tikzpicture}[scale = .5]
\draw[fill,gray!25] (0,0)--(2,0)--(2,-2)--(-2,-2);

\draw[fill] (0,0) circle(.1);

\draw[very thick,->] (0,0)--(-2,-2);
\draw[very thick,->] (0,0)--(2,0);
\draw[->,dotted] (0,0)--(2,0);
\draw[->,dotted] (0,0)--(-2,0);
\draw[->,dotted] (0,0)--(0,2);
\draw[->,dotted] (0,0)--(0,-2);

\end{tikzpicture}
& 
\begin{tikzpicture}
\draw[fill,gray!25] (0,0)--(1,1)--(1,-1);
\draw[thick] (0,0)--(1,1);
\draw[very thick] (1,1)--(1,-1);
\draw[thick] (0,0)--(1,-1);

\draw[fill = white] (1,1) circle(.1);
\draw[fill] (0,0) circle(.05);
\draw[fill] (1,-1) circle(.05);
\end{tikzpicture} \\

\raisebox{1.4cm}{$\mathbb{P}^4$} \rule{0pt}{2.8cm} 
& 
\begin{tikzpicture}[scale = .5]
\draw[fill,red!50] (-2,2)--(0,0)--(2,0)--(2,2);
\draw[very thick,red,->] (0,0)--(-2,2);
\draw[red,fill=white] (-1,1) circle(.2);
\draw[red, fill=red] (-1,1) circle(.1);
\draw[fill,gray!25] (0,0)--(2,0)--(2,-2)--(-2,-2);

\draw[fill] (0,0) circle(.1);

\draw[very thick,->] (0,0)--(-2,-2);
\draw[very thick,->] (0,0)--(2,0);
\draw[->,dotted] (0,0)--(-2,0);
\draw[->,dotted] (0,0)--(0,2);
\draw[->,dotted] (0,0)--(0,-2);

\end{tikzpicture}
& 
\begin{tikzpicture}
\draw[fill,gray!25] (0,0)--(1,1)--(1,-1);
\draw[thick] (0,0)--(1,1);
\draw[very thick] (1,1)--(1,-1);
\draw[thick] (0,0)--(1,-1);

\draw[red,fill = white] (1,1) circle(.1);
\draw[fill] (0,0) circle(.05);
\draw[fill] (1,-1) circle(.05);
\draw[red,fill = red] (1,1) circle(.05);
\end{tikzpicture} \\

\raisebox{1.4cm}{$\Bl_0(\mathbb{P}^4)$} \rule{0pt}{2.8cm}
&
\begin{tikzpicture}[scale = .5]
\draw[fill,gray!25] (0,0)--(2,0)--(2,2);
\draw[fill,gray!25] (0,0)--(2,0)--(2,-2)--(-2,-2);

\draw[fill] (0,0) circle(.1);

\draw[very thick,->] (0,0)--(2,2);
\draw[very thick,->] (0,0)--(-2,-2);
\draw[very thick,->] (0,0)--(2,0);
\draw[->,dotted] (0,0)--(2,0);
\draw[->,dotted] (0,0)--(-2,0);
\draw[->,dotted] (0,0)--(0,2);
\draw[->,dotted] (0,0)--(0,-2);

\end{tikzpicture}
& 
\begin{tikzpicture}
\draw[fill,gray!25] (0,0)--(.5,.5)--(1,0)--(1,-1);
\draw[very thick] (1,0)--(1,-1);
\draw[thick] (0,0)--(.5,.5);
\draw[thick] (.5,.5)--(1,0);
\draw[thick] (0,0)--(1,-1);

\draw[fill] (.5,.5) circle(.05);
\draw[fill] (1,0) circle(.05);
\draw[fill] (0,0) circle(.05);
\draw[fill] (1,-1) circle(.05);
\end{tikzpicture} \\
\hline
\end{tabular}
\caption{Colored fans and colored tropicalizations of $\Gl_2$-embeddings}
\label{coltrops}
\end{table}
\end{example}

\section{The Fundamental Theorem}\label{FTTG}

Along with Vogiannou's work, Kaveh and Manon in \cite{KM} give an alternate means for tropicalizing subvarieties of spherical homogeneous spaces.
Their approach uses a Gr\"obner theory for spherical varieties that they develop. 
Ultimately, they prove via a fundamental theorem that the two constructions coincide. 
In the following sections, we describe how their construction can be extended to spherical embeddings and show that this notion of extended spherical tropicalization coincides with that described in \textsection \ref{construction}.
We retain the conventions and notation used previously.

We begin by outlining the Fundamental Theorem of Tropical Geometry in the toric case (cf. \cite{SS}, Theorem 2.1 or \cite{MS}, Theorem 3.2.3).
The theorem is usually stated in terms of $\mathbb{R}$ rather than $\mathbb{Q}$.
We work over $\mathbb{Q}$ as this is the standard convention in spherical geometry; our results can be modified by tensoring with $\mathbb{R}$ to no ill effect.

Let $v: k \rightarrow \overline{\mathbb{Q}}$ be a valuation and assume that there exists a splitting $\varphi: v(k^\ast) \rightarrow k^\ast$ such that $(v \circ \varphi)(w) = w$.
We will write $t^w = \varphi(w)$.
Then we have an associated valuation ring $R := \set{x \in k \mid v(x) \geq 0}$ with maximal ideal $\mathfrak{m} := \set{x \in k \mid v(x) > 0}$.
We write the residue field $\Bbbk := R/\mathfrak{m}$.
Denote by $\overline{x}$ the image of $x \in R$ under the projection map $R \rightarrow \Bbbk$.
Let $f = \sum_{\vect{u} \in \mathbb{Z}^n} a_\vect{u}\vect{x}^\vect{u} \in k[x_1^{\pm 1},\ldots, x_m^{\pm 1}]$ and let $\vect{w} = (w_1,\ldots,w_m) \in \mathbb{Q}^{m}$ be an arbitrary vector.
We call $\vect{w}$ the weight vector.
Write
\[
W := \trop{f}(\vect{w}) = \min_{a_\vect{u} \neq 0}\set{v(a_\vect{u}) + \vect{u} \cdot \vect{w}}.
\]
\begin{definition}\label{initform}
The \emph{initial form} $\text{in}_\vect{w}(f) \in \Bbbk[x_1^{\pm 1}, \ldots, x_m^{\pm 1}]$ of $f := \sum_{\vect{u} \in \mathbb{Z}^m} a_\vect{u}\vect{x}^\vect{u}$ with respect to the weight vector $\vect{w}$ is:
\[
\text{in}_\vect{w}(f) := \sum_{\substack{\vect{u} \in \mathbb{Z}^m \\
v(a_\vect{u}) + \vect{u} \cdot \vect{w} = W}}
\overline{t^{-v(a_\vect{u})} a_\vect{u}} \vect{x}^\vect{u} = 
\overline{t^{-W} f(t^{w_1}x_1,\ldots,t^{w_m}x_m)}.
\]
The second characterization here is only valid when $\vect{w} \in (v(k^\ast))^m$ since otherwise the $t^{w_i}$ are not defined.
\end{definition}
\begin{example}\label{E3}
Let $f(x_1,x_2) = 2t + \left(t^{-1} + 3t^3 \right)x_1 + \left(7 - t^{1000} \right)x_2 - 6x_1^2 + 4t^{-2}x_1x_2 \in \mathbb{C}\{\{t\}\}\left[x_1^{\pm 1},x_2^{\pm 1} \right]$ so that $\trop{f}(\vect{u}) = \min\set{1,u_1 - 1,u_2,2u_1,u_1+u_2 - 2}$.
If $\vect{w}_1 = (-2,0)$, then $\trop{f}(\vect{w}) = -4$ and this value is met at the monomials $-6x_1^2$ and $4t^{-2}x_1x_2$, so the initial form is as follows:
\[
\text{in}_{\vect{w}}(f) = \overline{t^0(-6)}x_1^2 + \overline{t^2(4t^{-2})}x_1x_2 = -6x_1^2 + 4x_1x_2. 
\]
If instead we consider $\vect{w} = (0,2)$, then $\trop{f}(\vect{w}) = -1$ is met solely at the monomial $(t^{-1} + 3t^3)x_1$.
Thus, the initial form in this case is
\[
\text{in}_{\vect{w}}(f) = \overline{t(t^{-1} + 3t^3)}x_1 =  \overline{1 + 3t^4}x_1 = x_1 
\]
\end{example}

Now we can define the initial ideal of an ideal in $k[x_1^{\pm 1},\ldots,x_m^{\pm 1}]$:
\begin{definition}
The initial ideal of an ideal $I \subseteq k[x_1^{\pm 1},\ldots,x_m^{\pm 1}]$ with respect to $\vect{w} \in \mathbb{Q}^m$ is 
\[
\text{in}_\vect{w}(I) := \langle \text{in}_\vect{w}(f) \mid f \in I \rangle.
\]
\end{definition}
The construction of initial ideals is similar in spirit to the theory of Gr\"obner bases in a polynomial ring, but when we consider Laurent polynomials, any monomial is a unit in the ring. Thus, if $\text{in}_\vect{w}(f)$ is a unit for any $f \in I$, then $\text{in}_\vect{w}(I) = k[x_1^{\pm 1},\ldots,x_m^{\pm 1}]$. We are now able to state the Fundamental Theorem.
\begin{theorem} \label{FundThm} \emph{Fundamental Theorem of Tropical Geometry.} 
Let $k$ be an algebraically closed field with a nontrivial valuation $v$ and let $I \subseteq k[x_1^{\pm 1},\ldots,x_m^{\pm 1}]$ be an ideal with associated variety $V(I) = \set{\vect{x} \in (k^\ast)^m \mid f(\vect{x}) = 0 \text{ for all } f \in I}$ in $\left(k^* \right)^m$. Then the following subsets of $\mathbb{Q}^m$ coincide:
\begin{enumerate}
\item $\trop{V(I)} := \bigcap_{f \in I} \trop{V(f)}$;
\item $\set{\vect{w} \in \mathbb{Q}^m \mid \text{in}_\vect{w}(I) \neq k[x_1^{\pm 1},\ldots,x_m^{\pm 1}]}$;
\item The closure in $\mathbb{Q}^m$ of the set
\[
v(V(I)) := \set{(v(x_1),\ldots,v(x_m)) \mid (x_1,\ldots,x_m) \in V(I)}.
\]
\end{enumerate}
\end{theorem}
We have in this theorem the hypothesis that $v$ is nontrivial, but this is not particularly restrictive. In fact, if $I$ is a subvariety of a torus $(k^\ast)^m$ and $K$ is a valued extension of $k$ (i.e. $v_K|_k = v_k$), then the tropicalization of $I$ in $(k^\ast)^m$ is equal to the tropicalization of $I$ in $(K^\ast)^m$ (cf. \cite{MS}, Theorem 3.2.4).
Consider for example how the trivial valuation on $\mathbb{C}$ can be extended to the $t$-adic valuation on $\mathbb{C}\{\{t\}\}$.
As a final comment, it is worth mentioning that the work we did in Example \ref{E3} agrees with this theorem: the vector $(-2,0)$ lies on the tropical hyperplane and its initial form is a binomial, while the vector $(0,2)$ does not and its initial form is a monomial, which is a unit in $k[x_1^{\pm 1},\ldots,x_m^{\pm 1}]$.

We now describe how Theorem \ref{FundThm} can be extended to the tropicalization of toric varieties in the spirit of \cite{Ka} and \cite{Pay}.

If $f = \sum c_{\vect{u}}x^\vect{u} \in k[x_1,\ldots,x_m]$ and $\vect{w} \in \overline{\mathbb{Q}}^m$,
we define 
\[
\trop(f)(\vect{w}) := \min\set{v(c_{\vect{u}}) + \vect{w} \cdot \vect{u} : \vect{u} \in \left(\mathbb{N} \cup \set{0} \right)^m} 
\]
with the convention that if $w_i = \infty$, then $w_i \cdot u_i = \infty$ when $u_i \neq 0$ and $w_i \cdot u_i = 0$ when $u_i = 0$.
When $\trop(f)(\vect{w}) < \infty$, then we define the initial form $\init_\vect{w}(f)$ exactly as in Definition \ref{initform}. If $\trop(f)(\vect{w}) = \infty$, then we define $\init_{\vect{w}}(f) = 0$.
The ideal $\init_\vect{w}(I) = \langle \init_\vect{w}(f) : f \in I \rangle \in \Bbbk[x_1,\ldots,x_m]$ is defined as before.

We establish one final piece of notation before stating the theorem.
If $\sigma \subseteq \set{1,2,\ldots,n}$ and $\vect{w} \in \mathbb{Q}^{m-|\sigma|}$ is indexed by $\set{i : i \notin \sigma}$, then we write $\vect{w} \times \infty^\sigma \in \mathbb{Q}^m$ to be the vector that is $w_i$ when the coordinate $i \notin \sigma$ and $\infty$ otherwise.
If $\Sigma \subseteq \mathbb{Q}^{m - |\sigma|}$, then $\Sigma \times \infty^\sigma := \set{\vect{w} \times \infty^\sigma : \vect{w} \in \Sigma}$. Now we can state the theorem:

\begin{theorem}[cf. \cite{MS}, Theorem 6.2.15] Let $Y$ be a subvariety of $\mathbb{A}^m$, and let $I \subseteq k[x_1,\ldots,x_m]$ be its ideal. Then the following subsets of $\overline{\mathbb{Q}}^m = \trop(\mathbb{A}^m)$ coincide:
\begin{enumerate}
\item $\bigcap_{f \in I} \trop(V(f))$;
\item the set of vectors $\vect{w} \in \mathbb{Q}^m$ for which $\init_{\vect{w}}(I) \subseteq \Bbbk[x_1,\ldots,x_m]$ does not contain a monomial; and
\item the set
\[
\bigcup_{\sigma \subseteq \set{1,\ldots,m}} \trop(Y \cap O_\sigma) \times \infty^\sigma,
\]
where $O_\sigma = \set{\vect{x} \in \mathbb{A}^m : x_i = 0 \text{ for } i \in \sigma, \text{ and } x_j \neq 0 \text{ for } j \notin \sigma}$.
\end{enumerate}
\end{theorem}

This statement is restricted to subvarieties of affine space, but it can be extended to arbitrary toric varieties by tropicalizing a quotient of a subvariety of the Cox ring of the variety. This is described explicitly in \cite{MS}, Corollary 6.2.16.

\section{Spherical Gr\"obner tropicalization}\label{Grobtrop}

We now describe the Kaveh-Manon notion of the tropicalization of a spherical homogeneous space. 
We will give an overview of the relevant points of \textsection 4 of \cite{KM}, which is relatively self-contained.
There is a great deal of general theory on this subject from \cite{KM} that will be omitted.

Let $G/H$ be a spherical homogeneous space and let $B \leq G$ be a Borel subgroup with associated palette $\mathcal{D}$.
We denote the open orbit of $B$ in $G/H$ by $(G/H)_{B}$, and it is given as follows:
\[
(G/H)_{B} = (G/H) \setminus \bigcup_{D \in \mathcal{D}} D.
\]

For fixed $v \in \mathcal{V}$ and $a \in \mathbb{Q}$, define $k\left[(G/H)_{B}\right]_{v \geq a} := \set{f \in k\left[(G/H)_{B}\right] : v(f) \geq a}$ and similarly $k\left[(G/H)_{B}\right]_{v > a}$. We define a graded algebra $\gr_v(k\left[(G/H)_{B} \right])$ as follows:
\begin{equation}\label{associatedgraded}
\gr_v\left(k\left[(G/H)_{B}\right]\right) = \bigoplus_{a \in \mathbb{Q}} k\left[(G/H)_{B}\right]_{v \geq a}/k\left[(G/H)_{B}\right]_{v > a}.
\end{equation}
(We note that there is an oversight in \cite{KM} as to this definition: the direct sum is taken over $\mathbb{Q}_{\geq 0}$ rather than $\mathbb{Q}$; $\mathbb{Q}$ is the correct notion.)
Let $f \in k\left[(G/H)_{B}\right]$ and $v \in \mathcal{V}$ and write $v(f) = a$.
Then we define $\init_v(f)$ to be the quotient of $f$ in $k\left[(G/H)_{B}\right]_{v \geq a}/k\left[(G/H)_{B}\right]_{v > a}$.
If $J_B \subseteq k\left[(G/H)_{B}\right]$ is an ideal, we define 
\[
\init_v(J_B) := \left\langle \init_v(f) : f \in J_B \right\rangle \subseteq \gr_v\left( k \left[ (G/H)_{B} \right] \right).
\]
We write $\mathcal{V}_{B}(J_B)$ to denote the set of $v \in \mathcal{V}$ such that $\init_v(J_B) \neq \gr_v\left(k\left[(G/H)_{B}\right]\right)$.
\begin{definition}
Let $Y \subseteq G/H$ be a closed subvariety defined by an ideal $J \subseteq k[G/H]$. For each Borel subgroup $B \leq G$, let $J_B \subseteq k\left[(G/H)_{B}\right]$ denote the ideal that defines $Y \cap (G/H)_{B}$ as a subvariety of $(G/H)_{B}$.
Then we define the \emph{Gr\"obner tropicalization} of $Y$ to be
\[
\trop_G^{\gr}(Y) := \bigcup_{B} \mathcal{V}_{B}\left(J_B \right),
\]
where the union is indexed over all Borel subgroups $B$ of $G$.
\end{definition}
Kaveh and Manon show (\cite{KM}, Proposition 4.10) that in fact we need only take the union over a finite number of Borel subgroups of $G$.
The following theorem and example illustrate the necessity of taking a union of multiple Borel subgroups.
Ultimately, we want to show that Gr\"obner tropicalization agrees with Vogiannou's tropicalization, and valuations lying in Vogiannou's tropicalization can be missed if not enough Borel subgroups are used.

\begin{theorem}[\cite{KM}, Theorem 4.6]\label{wheninGrob}
Let $Y \subseteq (G/H)_{B}$ be a subvariety defined by an ideal $J \subseteq k \left[(G/H)_{B}\right]$. Let $v \in \mathcal{V}$ be a valuation, let $X$ be the $G/H$-embedding whose colored fan consists of the ray spanned by $v$, and let $\mathcal{O}$ be the closed $G$-orbit in $X$.
Then $v$ lies in $\mathcal{V}_B(J)$ if and only if the closure of $Y$ in $X$ intersects the open $B$-orbit of $\mathcal{O}$.
\end{theorem}

\begin{example}[cf. \cite{Ga}, Example 4.3]
Let $G = \Sl_2$, $H$ be the diagonal torus, and $B$ be the upper triangular matrices.
Then $G/H \cong \left(\mathbb{P}^1 \times \mathbb{P}^1 \right) \setminus \mathcal{O}$, where $\mathcal{O}$ denotes the diagonal and $G$ acts naturally on each copy of $\mathbb{P}^1$.
The map $G \rightarrow \left(\mathbb{P}^1 \times \mathbb{P}^1 \right) \setminus \mathcal{O}$ is given as follows; its kernel is $H$:
\[
\left(\begin{array}{cc}
x_{11} & x_{12} \\
x_{21} & x_{22}
\end{array} \right) \mapsto
\left(\left[\begin{array}{c} x_{11} \\ x_{21} \end{array} \right],
\left[\begin{array}{c} x_{12} \\ x_{22} \end{array} \right] \right).
\]
Then the palette $\mathcal{D}$ consists of two colors: $V(x_{21})$ and $V(x_{22})$.
The $B$ semi-invariant rational functions are integer powers of $g := (x_{11}x_{22} - x_{12}x_{21})/x_{21}x_{22}$ and the lattice $\mathcal{N}$ is isomorphic to $\mathbb{Z}$. We claim $\mathcal{V}$ is the ray generated by the $G$-invariant valuation 
\[
f \mapsto \text{order of vanishing of $f$ along $\mathcal{O}$}.
\]
Note that $g \mapsto 1$ under this valuation.
We show that no valuation sending $g$ to a negative number can be $G$-invariant.
Note first that
\[
\left(\begin{array}{cc}
0 & 1 \\
-1 & 0
\end{array} \right) \cdot \frac{x_{22}}{x_{12}} = -\frac{x_{12}}{x_{22}} \quad \text{and} \quad \left(\begin{array}{cc}
0 & 1 \\
-1 & 0
\end{array} \right) \cdot \frac{x_{21}}{x_{11}} = -\frac{x_{11}}{x_{21}}
\]
This implies any $G$-invariant valuation must take both $\frac{x_{12}}{x_{22}}$ and $\frac{x_{11}}{x_{21}}$ to 0.
Thus, if $v$ is a $G$-invariant valuation, we have that
\[
v(g) = v\left(\frac{x_{11}x_{22} - x_{12}x_{21}}{x_{21}x_{22}} \right) = v\left( \frac{x_{11}}{x_{21}} - \frac{x_{12}}{x_{22}} \right) \geq \min\set{v\left(\frac{x_{11}}{x_{21}} \right), v\left(\frac{x_{12}}{x_{22}} \right)} = 0.
\]
Therefore, any $G$-invariant valuation must take $g$ to a non-negative number.

Both colors map to the same point outside the valuation cone and the only $G/H$-embedding is $X \cong \mathbb{P}^1 \times \mathbb{P}^1$, whose colored fan is the ray $\mathcal{V}$.
The data from the homogeneous space is shown in Figure \ref{GrobTropEx}.
\begin{center}
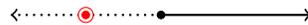
\begin{figure}[h]
\begin{tikzpicture}
\draw[fill] (0,0) circle(.05);

\draw[->,dotted,thick] (0,0)--(-2,0);
\draw[->,thick] (0,0)--(2,0);

\draw[red,fill = white] (-1,0) circle(.1);
\draw[red,fill = red] (-1,0) circle(.05);
\end{tikzpicture}
\caption{Spherical data for $\left( \mathbb{P}^1 \times \mathbb{P}^1 \right) \setminus \mathcal{O}$; $\mathcal{V}$ is the ray to the right and both colors lie on the left}\label{GrobTropEx}
\end{figure}
\end{center}

We have that
\begin{align*}
(G/H)_{B} & = (G/H) \setminus \set{V(x_{21}),V(x_{22})} \\
& = \set{\left([x_{11} : x_{21} ], [x_{12} : x_{22} ]\right) : x_{21},x_{22},x_{11}x_{22} - x_{12}x_{21} \neq 0} \\
& =  \set{([x_{11} : 1 ], [x_{12} : 1 ]) : x_{11} \neq x_{12}} \\
& \cong \mathbb{A}^2 \setminus \Delta.
\end{align*}
Here, $\Delta$ represents the diagonal of $\mathbb{A}^2$.
This tells us that $k\left[(G/H)_{B} \right] \cong k[x_{11},x_{12},(x_{11}-x_{12})^\pm]$.

Let us consider $J = (x_{21})$ and the corresponding subvariety $Y := V(x_{21}) \subset G/H$. 
The closure of $Y$ in $X$ intersects $\mathcal{O}$ in the point $([1:0],[1:0])$, which does not lie in the open $B$-orbit of $\mathcal{O}$.
Thus, Theorem \ref{wheninGrob} tells us $\gr_B(J_B)$ consists solely of the trivial valuation since it can contain no other valuation along the ray $\mathcal{V}$.

If instead we had used the Borel subgroup $B'$ of lower triangular matrices, then the resulting palette $\mathcal{D}'$ would consist of $V(x_{11})$ and $V(x_{12})$ and we would still have $(G/H)_{B'} \cong \mathbb{A}^2 \setminus \Delta$.
The important difference is that in this case we have:
\[
(G/H)_{B'} = \set{([1 : x_{21} ], [1 : x_{22} ]) : x_{21} \neq x_{22}}.
\]
The subvariety $Y$ still intersects the closed $G$-orbit of $X$ in the point $([1:0],[1:0])$, but this point lies in the open $B'$-orbit.
Thus, $\gr_{B'}(J_{B'})$ consists of the valuation cone $\mathcal{V}$, which is the tropicalization of $V(x_{21})$.
In general, if $Y \subseteq G/H$ is cut out by an ideal $J$, we have $\trop_G^{\gr}(Y) = \mathcal{V}_B(J_B) \cup \mathcal{V}_{B'}(J_{B'})$.
\end{example}

This notion of spherical tropicalization via Gr\"obner theory agrees with Vogiannou's tropicalization:
\begin{theorem}[\cite{KM}, Theorem 5.15]\label{fundthmspherical}
Let $Y \subseteq G/H$ be a closed subvariety.
Then Gr\"obner tropicalization coincides with spherical tropicalization:
\[
\trop_G(Y) = \trop_G^{\gr}(Y).
\]
\end{theorem}

\section{Extended Gr\"obner tropicalization}\label{extGrobtrop}

In this section we extend Gr\"obner tropicalization to encompass spherical embeddings. 
Let $X$ be a simple spherical embedding defined by a colored cone $\left( \sigma, \mathcal{F} \right)$.
Write $\mathcal{O}$ for the unique closed $G$-orbit of $X$, write $\mathcal{D}(X)$ for the set of $B$-stable prime divisors of $X$, and write $\mathcal{D}_\mathcal{O}(X)$ for the $B$-stable prime divisors of $X$ that contain $\mathcal{O}$.
There is a $B$-stable subset $X_{B}$ of $X$ defined by:
\[
X_{B} := X \setminus \bigcup_{D \in \mathcal{D}(X)\setminus \mathcal{D}_\mathcal{O}(X)} D,
\]
in other words by throwing out the $B$-stable prime divisors not containing $\mathcal{O}$.
Note that this is consistent with the earlier notation $(G/H)_B$ and that $(G/H)_B = X_B \cap (G/H)$.
This theory appears with more detail in \cite{Kn} \textsection 2, where the notation $X_0$ is used rather than $X_{B}$.
The following is a straightforward extension of Theorem 4.1 in \cite{KM}.
\begin{theorem}\label{orbitregfunctions}
Let $X$ be a simple $G/H$-embedding with closed orbit $\mathcal{O}$.
The subset $X_{B}$ is open and affine, and $\mathcal{O} \cap X_{B}$ is a $B$-orbit. Further, the regular functions on $X_{B}$ can be described as follows:
\[
k\left[X_{B}\right] = \set{f \in k\left[(G/H)_{B}\right] : \nu_D(f) \geq 0 \text{ for all } D \in \mathcal{D}_\mathcal{O}(X) }. 
\]
\end{theorem}
\begin{proof}
The first claims are from \cite{Kn}, Theorem 2.1. Turning to the regular functions, note that because $(G/H)_{B}$ is open in $X_{B}$, we can identify a function in $k\left[X_{B}\right]$ with its restriction to $(G/H)_{B}$.
The functions in $k\left[(G/H)_{B}\right]$ that are restrictions in this way are precisely those that do not have poles when evaluated at points in $X_{B} \setminus (G/H)_{B}$.
In other words, they are the functions $f$ such that $\nu_D(f) \geq 0$ for all $D \in \mathcal{D}_\mathcal{O}(X)$.
Conversely, if we have a function in $k\left[(G/H)_{B}\right]$, it extends (uniquely) to $k\left[X_{B}\right]$ precisely when it does not have poles along any $D \in \mathcal{D}_\mathcal{O}(X)$.
\end{proof}

The orbit $\mathcal{O}$ is a spherical homogeneous space under the action of $G$ with open $B$-orbit $\mathcal{O} \cap X_{B}$.
It follows that $\mathcal{D}(\mathcal{O}) = \set{D \cap \mathcal{O} : D \in \mathcal{D}(X) \setminus \mathcal{D}_\mathcal{O}(X)}$.

Now denote by $\mathcal{V}\left(X_{B} \right)$ the extended $G$-invariant $\overline{\mathbb{Q}}$-valuations on $k\left[(G/H)_{B}\right]$ that are finite on $k\left[X_{B}\right]$:
\[
\mathcal{V}\left(X_{B} \right) := \set{v: k\left[(G/H)_{B}\right] \rightarrow \overline{\mathbb{Q}} : k\left[X_{B}\right] \subseteq v^{-1}(\mathbb{Q}) }.
\]
Observe that $\mathcal{V}\left(X_{B} \right)$ can be identified with a subset of $\overline{\mathcal{N}}_\mathcal{V}(\sigma)$.
Indeed, a $G$-invariant $\overline{\mathbb{Q}}$-valuation on $k\left[(G/H)_{B}\right]$ induces a semigroup homomorphism $\sigma^\vee \cap \mathcal{M} \rightarrow \overline{\mathbb{Q}}$ in the same way a $G$-invariant $\overline{\mathbb{Q}}$-valuation on $k\left[G/H\right]$ does.
The only difference is that valuations on $k\left[(G/H)_{B}\right]$ cannot take infinite values on functions cutting out the $B$-stable divisors because these functions are invertible in $k\left[(G/H)_B \right]$.

For any $v \in \mathcal{V}\left(X_{B} \right)$, we define $\gr_v\left(k\left[X_{B}\right]\right)$ as follows (cf. Equation (\ref{associatedgraded})).
\begin{equation}\label{extended grobner def}
\gr_v\left(k\left[X_{B}\right]\right) = k\left[(G/H)_{B}\right]_{v = \infty} \oplus \bigoplus_{a \in\mathbb{Q}} k\left[(G/H)_{B}\right]_{v \geq a}/k\left[(G/H)_{B}\right]_{v > a},
\end{equation}
where $k\left[(G/H)_{B}\right]_{v = \infty} := \set{f \in k\left[(G/H)_B \right] : v(f) = \infty}$.
If we assert $k\left[(G/H)_{B}\right]_{v > \infty} := \set{0}$, then we may write
\[
\gr_v\left(k\left[X_{B}\right]\right) = \bigoplus_{a \in \overline{\mathbb{Q}}} k\left[(G/H)_{B}\right]_{v \geq a}/k\left[(G/H)_{B}\right]_{v > a}.
\]
Note that if $v \in \mathcal{V}$, the first summand of Equation (\ref{extended grobner def}) vanishes and $\gr_v\left(k\left[X_{B}\right]\right) = \gr_v\left(k\left[G/H_{B}\right]\right)$.
For $f \in k\left[X_{B}\right]$, $J \subseteq k\left[X_{B}\right]$, and $v \in \mathcal{V}(X_B)$, we define $\init_v(f)$ and $\init_v(J)$ as before.
Similarly, $\mathcal{V}_B(J)$ is the set of $v \in \mathcal{V}(X_B)$ such that $\init_v(J) \neq \gr_v\left(k\left[X_{B}\right]\right)$. 
Now we can define the extended Gr\"obner tropicalization of a subvariety of a spherical embedding:
\begin{definition}
Let $X$ be a simple $G/H$-embedding associated to a colored cone $\left(\sigma, \mathcal{F} \right)$ and $Y \subseteq G/H$ a closed subvariety defined by an ideal $J \subseteq k[G/H]$. Let $\overline{Y} \subseteq X$ denote the closure of $Y$ in $X$. For each Borel subgroup $B \leq G$, let $\overline{J}_B \subseteq k\left[X_{B}\right]$ denote the ideal that defines $\overline{Y} \cap X_{B}$ as a subvariety of $X_{B}$.
The \emph{extended Gr\"obner tropicalization} of $\overline{Y}$ in $\overline{\mathcal{N}}_\mathcal{V}(\sigma)$ is:
\[
\trop_G^{\gr}\left( \overline{Y} \right) := \bigcup_{B} \mathcal{V}_{B}\left( \overline{J}_B \right).
\]
In particular, if $Y = G/H$, then
\[
\trop_G^{\gr}(X) = \bigcup_B \mathcal{V}_B(0).
\]
If $X$ is a non-simple $G/H$-embedding, we define $\trop_G^{\gr}\left( \overline{Y} \right)$ as follows. For each simple $G/H$-embedding $X' \subseteq X$, compute $\trop_G^{\gr}\left( \overline{Y} \cap X' \right) \subseteq \trop_G^{\gr}(X')$ and glue together along shared valuations.  
\end{definition}

We finally prove an extended fundamental theorem equating our two notions of extended spherical tropicalization.
We first prove the statement for simple embeddings and then deduce the full result as a corollary.

\begin{theorem}\label{TheFinalThm}
Let $X$ be a simple $G/H$-embedding and $Y \subseteq G/H$ be a closed subvariety. Then 
\[
\trop_G(\overline{Y}) = \trop_G^{\gr}(\overline{Y})
\]
as subspaces of $\overline{\mathcal{N}}_\mathcal{V}(\sigma)$, where the closure $\overline{Y}$ is taken in $X$.
\end{theorem}
\begin{proof}
We first show that $\trop_G(X) = \trop_G^{\gr}(X)$.
Let $\left( \sigma, \mathcal{F} \right)$ be the colored cone associated to $X$ and let $\mathcal{O}$ be the $G$-orbit associated to a colored face $\tau \preceq \sigma$.
If $B \leq G$ is a Borel subgroup, denote by $\mathcal{D}_\mathcal{O}(X) \subseteq \mathcal{D}(X)$ the set of $B$-stable divisors of $X$ that contain $\mathcal{O}$.
Write
\[
X_{B}(\mathcal{O}) := X \setminus \bigcup_{D \in \mathcal{D}(X)\setminus \mathcal{D}_\mathcal{O}(X)} D \subseteq X_{B}
\]
and note that $X_{B}(\mathcal{O}) = X_{B}$ when $\mathcal{O}$ is the unique closed orbit of $X$.
We can characterize the regular functions on $X_{B}(\mathcal{O})$ as we did in Theorem \ref{orbitregfunctions}:
\[
k\left[X_{B}(\mathcal{O})\right] = \set{f \in k\left[(G/H)_{B}\right] : \nu_D(f) \geq 0 \text{ for all } D \in\mathcal{D}_\mathcal{O}(X) }. 
\]
Now consider the following subset:
\[
\set{v: k\left[(G/H)_{B}\right] \rightarrow \overline{\mathbb{Q}} : k\left[X_{B}(\mathcal{O})\right] = v^{-1}(\mathbb{Q}) } \subseteq \mathcal{V}\left(X_{B}\right) \subseteq \overline{\mathcal{N}}_\mathcal{V}(\sigma).
\]
The condition that $k\left[ X_B(\mathcal{O}) \right] = v^{-1}(\mathbb{Q})$ is the same as the condition that $v$ satisfies $v^{-1}(\mathbb{Q}) = \tau^\perp \cap \left( \sigma^\vee \cap \mathcal{M} \right)$, with the additional requirement that  functions vanishing along the $B$-stable prime divisors aren't sent to $\infty$.
Taking the union over all Borel subgroups thus gives a copy of $\mathcal{V}(\mathcal{O}) \subseteq \Hom\left( \tau^\perp \cap \left(\sigma^\vee \cap \mathcal{M} \right) , \overline{\mathbb{Q}} \right)$ in $\trop_G^{\gr}(X)$.
This holds for all $G$-orbits $\mathcal{O}$, so $\trop_G^{\gr}(X) = \trop_G(X)$.

Now if $Y \subseteq G/H$ is a closed variety, suppose $\overline{Y} \cap \mathcal{O} \neq \emptyset$ for some $G$-orbit $\mathcal{O}$.
Then Theorem \ref{fundthmspherical} ensures that $\trop_G\left(\overline{Y} \cap \mathcal{O} \right) = \trop_G^{\gr}\left( \overline{Y} \cap \mathcal{O} \right)$ in $\mathcal{V}(\mathcal{O}) \subseteq \overline{\mathcal{N}}_\mathcal{V}(\sigma)$.
As this holds for every orbit $\mathcal{O}$, the statement follows.
\end{proof}

\begin{corollary}
Let $X$ be a $G/H$-embedding and $Y \subseteq G/H$ a closed subvariety. Then 
\[
\trop_G(\overline{Y}) \cong \trop_G^{\gr}(\overline{Y}),
\]
where the closure $\overline{Y}$ is taken in $X$.
\end{corollary}
\begin{proof}
Theorem \ref{TheFinalThm} proves this for simple embeddings. The gluing operation between tropicalizations of simple embeddings is identical for both constructions, so the result follows.
\end{proof}

\end{document}